\documentclass[10pt]{amsart}
\usepackage{amssymb,mathrsfs}
\usepackage[utf8]{inputenc}
\usepackage[T1]{fontenc}
\usepackage{lmodern,enumitem}

\def\R{\mathbb{R}}
\def\N{\mathbb{N}}
\def\C{\mathbb{C}}

\def\scrL{\mathscr{L}}

\newcommand{\e}{\mathrm{e}}

\newcommand{\Span}{\operatorname{span}}

\newcommand{\vertiii}[1]{{\left\vert\kern-0.25ex\left\vert\kern-0.25ex\left\vert #1 
    \right\vert\kern-0.25ex\right\vert\kern-0.25ex\right\vert}}

\newcommand{\ba}{\begin{aligned}}
\newcommand{\ea}{\end{aligned}}

\newcommand{\be}{\begin{equation}}
\newcommand{\ee}{\end{equation}}

\newtheorem{theorem}{Theorem}
\newtheorem{corollary}[theorem]{Corollary}

\newtheorem{lemma}[theorem]{Lemma}
\newtheorem{proposition}[theorem]{Proposition}

\theoremstyle{definition}

\begin{document}

\title{Spectral description of a cell growth and division equation}

\author[P. Gabriel]{Pierre Gabriel}
\address[P. Gabriel]{Institut Denis Poisson, Université de Tours, Université d’Orléans, CNRS, Tours, France}
\email{pierre.gabriel@univ-tours.fr}
\author[B. van Brunt]{Bruce van Brunt}
\address[B. van Brunt]{School of Mathematical and Computational Sciences, Massey University, Auckland, New Zealand}
\email{B.vanBrunt@massey.ac.nz}
\author[G. C. Wake]{Graeme Charles Wake}
\address[G. C. Wake]{School of Mathematical and Computational Sciences, Massey University, Auckland, New Zealand}
\email{G.C.Wake@massey.ac.nz}
\author[A. A. Zaidi]{Ali Ashher Zaidi}
\address[A. A. Zaidi]{Department of Mathematics, Lahore University of Management Sciences, Lahore, Pakistan}
\email{ali.zaidi@lums.edu.pk}

\begin{abstract}
We give a refined description of the dominant spectrum of a non-local operator that models growth and equal mitosis of cells.
More precisely we look at the spectrum in half planes at the right hand side of the first accumulation point of eigenvalues
and give criteria on the weight of weighted $L^1$ spaces for this spectrum to be made of explicit simple eigenvalues.
The method relies on a high order long time asymptotic expansion of the solutions to the associated evolution equation obtained in [Zaidi, van Brunt, Wake, Proc. A, R. Soc. Lond., 2015]
combined with a Weyl theorem taken from [Mischler, Scher, Ann. Inst. Henri Poincaré, Anal. Non Linéaire, 2016].
\end{abstract}



\maketitle


 \section{Introduction}
 
We are interested in the following non-local partial differential equation
\begin{equation}\label{eq:mitosis}
\partial_tu(t,x)+g\,\partial_xu(t,x)+b\,u(t,x)=4b\,u(t,2x)
\end{equation}
set for $x>0$ and complemented with zero flux boundary condition $u(t,0)=0$.
This equation appears as a population model of cell growth and division,
see~\cite{Bell1968,Bell1967,Sinko1967} for historical references but also the books~\cite{Banasiak2020,MetzDiekmann,Perthame}.
Within this modelling, $x$ is the size (volume, mass, molecular content...) of the cells, which grows with a constant speed $g>0$ and divide into two half sized daughter cells with a constant rate $b>0$,
and $u(t,x)$ represents the size distribution of the cells in the population at time $t$.

\

Integrating Equation~\eqref{eq:mitosis} in $x$, we get that
\[N(t):=\int_0^\infty u(t,x)dx\]
the total number of cells in the population, verifies the simple differential equation $N'=bN$.
It is thus explicitly given by
\begin{equation}\label{eq:N}
N(t)=N_0\e^{bt},\quad N_0=\int_0^\infty u_0(x)dx,\quad u_0(x)=u(0,x).
\end{equation}
A natural question is then the existence of solutions with steady size distribution, or more precisely the existence of a function $f_0:(0,\infty)\to[0,\infty)$, $f_0\not\equiv0$, such that
\[(t,x)\mapsto \e^{bt}f_0(x)\]
satisfies~\eqref{eq:mitosis}.
Such a steady size distribution $f_0$ was discovered in~\cite{HallWake89}.
It is given explicitly as a series, which already appeared in~\cite{KatoMcLeod} in a different context and was also recovered more recently in~\cite{PerthameRyzhik}.

The function $f_0$ can also be seen as an eigenfunction of the functional differential operator $\scrL$ defined by
\[\scrL f(x)=-g f'(x)-bf(x)+4bf(2x)\]
associated to the eigenvalue $b$, namely
\[\scrL f_0=bf_0.\]
Higher eigenfunctions were exhibited in~\cite{vBVH1,vBVH2} in the form of Dirichlet series,
associated to the eigenvalues
 \[\lambda_m=(2^{1-m}-1)b\]
 with $m$ nonnegative integers.
Denoting by $(f_m)_{m\geq0}$ this family, we thus have
\[\scrL f_m=\lambda_m f_m.\]

A challenging further issue is to identify the whole spectrum of $\scrL$ in suitable Banach spaces.
This question is strongly related to the long-time asymptotic behaviour of the solutions to~\eqref{eq:mitosis}.
In~\cite{Zaidi2015} it was proved that the eigenfunctions $(f_m)$ provide the pointwise asymptotic expansion of the solutions to~\eqref{eq:mitosis} when $t\to+\infty$.
More precisely, under continuity assumption on the initial distribution $u_0$, the solution satisfies, for any $x>0$ and $M\in\N$,
\begin{equation}\label{eq:expansion}
u(t,x)=\sum_{m=0}^M \alpha_m \e^{\lambda_mt}f_m(x)+O\big(e^{\lambda_{M+1}t}\big)\qquad\text{as}\ t\to+\infty
\end{equation}
for some real numbers $\alpha_m$ that depend only on $u_0$.
In the present paper, we aim at strengthening this result into a (uniform) convergence in norm in some weighted $L^1$ spaces,
and inferring a description of the dominant spectrum of $\scrL$ (with zero flux boundary condition) in these spaces.

\

We work in the spaces $L^1_k=L^1((0,\infty),\varpi_k(x)dx)$, where $\varpi_k(x)=(1+x)^k$ with $k\in\R$.
The dual space of $L^1_k$ is
\[(L^1_k)'=L^\infty_k:=\big\{\phi\in L^\infty_{loc}(0,\infty),\ \phi\varpi_{-k}\in L^\infty(0,\infty)\big\},\]
endowed with the norm $\|\phi\|_{L^\infty_k}=\|\phi\varpi_{-k}\|_{L^\infty}$.
All the eigenfunctions $f_m$ belong to $L^1_\infty=\bigcap_{k\geq0}L^1_k$
and we have the following result.

 \begin{theorem}\label{thm:main}
Let $a\in(-b,b)$. For any $k>\max(1,k_a)$ where
\[k_a=\frac{\log(2b)-\log(b+a)}{\log2},\]
we have for all $u_0\in L^1_k$ and all $t\geq0$
\begin{equation}\label{eq:higher-order}
\bigg\| u(t,\cdot)-\sum_{m=0}^{m_a}\langle \phi_m,u_0\rangle\,\e^{\lambda_m t}f_m\bigg\|_{L^1_k}\leq C\e^{a t}\bigg\| u_0-\sum_{m=0}^{m_a}\langle \phi_m,u_0\rangle f_m\bigg\|_{L^1_k}
\end{equation}
for some dual eigenfunctions $(\phi_m)\in L^\infty_k$,
where $m_a=\lceil k_a-1\rceil$ is the integer such that $\lambda_{m_a+1}\leq a<\lambda_{m_a}$.
 \end{theorem}
 
 This theorem has implications on the spectrum of $\scrL$ seen as an unbounded operator in $L^1_k$ with domain
 \[D_k(\scrL)=\big\{f\in L^1_k\,|\ f'\in L^1_k\ \text{and}\ f(0)=0\big\}.\]
Recall that the spectrum of $\scrL$ in $L^1_k$ is defined as
\[\sigma_k(\scrL)=\big\{\lambda\in\C\,|\ \lambda-\scrL:D_k(\scrL)\to L^1_k\ \text{is not bijective}\big\}.\]
Since $f_m\in L^1_\infty$, we have that $(\lambda_m)_{m\geq0}\subset\sigma_k(\scrL)$ for any $k\in\R$,
but the rest of the spectrum may depend on the value of $k$.
For $a\in\R$ we define
\[\Delta_a:=\{z\in\C\,|\ \mathrm{Re}(z)>a\}.\]
The following result is a classical consequence of Theorem~\ref{thm:main}.

 \begin{corollary}
 For any $a\in(-b,b)$ and $k>\max(1,k_a)$ we have
 \[\sigma_k(\scrL)\cap\Delta_a=\{\lambda_0,\cdots,\lambda_{m_a}\}.\]
 Besides, all the eigenvalues $\lambda_m$ are algebraically simple, in the sense that
 \[N((\lambda_m-\scrL)^j)=N(\lambda_m-\scrL)=\Span(f_m)\]
 for all integer $j$.
 \end{corollary}
 
 \
 
 Restricting~\eqref{eq:higher-order} to the first order, Theorem~\ref{thm:main} allows quantifying the Malthusian asymptotic behaviour
 \begin{equation}\label{eq:malthus}
 u(t,\cdot)\sim\langle\phi_0,u_0\rangle\, \e^{bt}f_0\qquad\text{as}\ t\to+\infty.
 \end{equation}
 More precisely, by~\eqref{eq:malthus} we mean that
 \begin{equation}\label{eq:malt-resc}
\e^{-bt} u(t,\cdot)\xrightarrow[t\to+\infty]{}\langle\phi_0,u_0\rangle\,f_0.
 \end{equation}
 Because of~\eqref{eq:N}, we see that if~\eqref{eq:malt-resc} holds, and if $f_0$ is normalized by $\|f_0\|_{L^1}=1$, then necessarily $\phi_0=\mathbf 1$, the constant function equal to $1$.
 A step further consists in establishing that the convergence~\eqref{eq:malt-resc} occurs exponentially fast, in the sense that
 \begin{equation}\label{eq:first-order}
 \big\| \e^{-bt}u(t,\cdot)-\langle \phi_0,u_0\rangle\,f_0\big\|_{L^1_k}\leq C\e^{-\omega t}\big\| u_0-\langle \phi_0,u_0\rangle f_0\big\|_{L^1_{k'}}
 \end{equation}
 for some constants $C,\omega>0$ and some exponents $k'\geq k\geq0$.
 Note that the existence of $f_1$, which is a steady state of the equation since $\lambda_1=0$, implies that necessarily $\omega\leq b$.
 Let us review the existing results in the literature about this convergence problem:
 \begin{enumerate}[label=(\roman*),leftmargin=*,itemsep=1mm]
 \item In~\cite{Michel2005}, the convergence~\eqref{eq:malt-resc} is proved to hold in~$L^1_0=L^1$ for any $u_0\in L^1$ as a consequence of the general relative entropy principle.
 This method does not provide any information on the speed of convergence.
 \item A direct consequence of~\cite[Theorem~1.1]{PerthameRyzhik} is that~\eqref{eq:first-order} holds true with $k=0$, $k'=1$, $\omega=b$ and $C=\max(1,6b)$.
 This result was then slightly improved in~\cite[Proposition~5.4]{Bardet2013} which allows taking $C=\max(1,2b)$, still with $k=0$, $k'=1$ and $\omega=b$.
 \item In~\cite[Proposition~6.5]{Mischler2016}, it is proved that for any $k>1+\frac{\log3}{\log2}$, the estimate~\eqref{eq:first-order} holds true with $k'=k$, any $\omega<(1-3\times2^{1-k})b$, and some $C=C(\omega)>0$.
 \item A negative result is provided in~\cite[Theorem~6.1]{Bernard2017}, where it is proved that for $k'=k>0$ and $\omega>(2\e k\log2)b$ there does not exist $C>0$ such that~\eqref{eq:first-order} holds.
 \end{enumerate}
 Theorem~\ref{thm:main} with $a=0$ improves the point (iii) since it ensures that~\eqref{eq:first-order} holds true with $k=k'>1$ and $\omega=b$.

 \begin{corollary}
 For any $k>1$ there exists $C>0$ such that
 \[\big\| \e^{-bt}u(t,\cdot)-\langle \phi_0,u_0\rangle\,f_0\big\|_{L^1_k}\leq C\e^{-b t}\big\| u_0-\langle \phi_0,u_0\rangle f_0\big\|_{L^1_{k}}\]
 for any $u_0\in L^1_k$ and all $t\geq0$.
 \end{corollary}

 \section{Primal and dual eigenfunctions}\label{sec:eigenfunctions}
 
 We start by recalling the existence result of solutions to the eigenvalue problem $\scrL f=\lambda f$ taken from~\cite{vBVH2,vBVH1}
 before investigating its dual counterpart $\scrL^*\phi=\lambda\phi$, where $\scrL^*$ is the formal adjoint operator of $\scrL$ given by
 \[\scrL^*\phi(x)=g\phi'(x)-b\phi(x)+2b\phi(x/2).\]
 
 \begin{proposition}[\cite{vBVH2,vBVH1}]
 For any nonnegative integer $m$, the function
 \[f_m(x)=\e^{-\frac bg 2^{1-m}x}+\sum_{n=1}^\infty\frac{(-1)^n 2^{n(m+1)}\e^{-\frac bg 2^{n+1-m}x}}{\prod_{j=1}^n(2^j-1)}\]
 satisfies $\scrL f_m=\lambda_m f_m$ where $\lambda_m=(2^{1-m}-1)b$.
 Moreover, $f_m\in L^1_\infty=\bigcap_{k\geq0}L^1_k$ and for any nonnegative integer $n$ we have
 \begin{equation}\label{eq:moments}
 \int_0^\infty x^nf_m(x)\,dx\neq0\quad\Longleftrightarrow\quad n\geq m.
 \end{equation}
 \end{proposition}
 
 In a similar but much simpler way, we can find dual eigenfunctions $\phi_m$ associated to $\lambda_m$ in the form of polynomials of degree $m$.
 
 \begin{lemma}
 Let $m$ be a nonnegative integer.
 If $m=0$, then any constant function $\phi_0$ verifies $\scrL^*\phi_0=\lambda_0\phi_0$.
 If $m>0$, then
 the function
 \[\phi_m(x)=\sum_{n=0}^m\alpha_{n,m}x^n\]
verifies $\scrL^*\phi_m=\lambda_m\phi_m$ if and only if
 \[\alpha_{n+1,m}=\frac{b}{g}\frac{2^{1-m}(1-2^{m-n})}{n+1}\alpha_{n,m}\]
 for any $0\leq n\leq m-1$.
 \end{lemma}
 
 The proof of this lemma is straightforward computations left to the reader.
 For any choice of $\alpha_{0,m}\neq0$, we thus have an eigenfunction $\phi_m$ of $\scrL^*$ for the eigenvalue $\lambda_m$.
 The property~\eqref{eq:moments} readily ensures that such a dual eigenfunction satisfies
 \[\langle\phi_m,f_m\rangle=\alpha_{m,m}\langle x^m,f_m\rangle\neq0.\]
 It is therefore possible, and we will make this choice from now on, to choose $\alpha_{0,m}$ such that
 \[\langle\phi_m,f_m\rangle=1.\]
 Since for $n\neq m$ we have 
 \[\lambda_n\langle\phi_n,f_m\rangle=\langle\scrL^*\phi_n,f_m\rangle=\langle\phi_n,\scrL f_m\rangle=\lambda_m\langle\phi_n,f_m\rangle\]
 and $\lambda_n\neq\lambda_m$, we deduce that
 \begin{equation}\label{eq:orthonorm}
 \langle\phi_n,f_m\rangle=\delta_{n,m}
 \end{equation}
 where $\delta_{n,m}$ is the Kronecker delta.

 \section{Proof of the main result}
 
 It is a known result, see for instance~\cite{Bernard2017}, that $\scrL$ generates a positive strongly continuous semigroup $(S_t)_{t\geq0}$ in $L^1_k$ for any $k\geq0$.
 This means that, for any $u_0\in L^1_k$, Equation~\eqref{eq:mitosis} admits a unique solution given by
 \[u(t,\cdot)=S_t u_0.\]
In particular we have $S_tf_m=\e^{\lambda_m t}f_m$ and $S^*_t\phi_m=\e^{\lambda_m t}\phi_m$ for all $m$, where $(S^*_t)$ is the dual semigroup of $(S_t)$ in $(L^1_k)'=L^\infty_k$.
 The proof of Theorem~\ref{thm:main} relies on a Weyl theorem which is proved in~\cite{Mischler2016} to hold for the semigroup $(S_t)$.
 
 \begin{proof}[Proof of Theorem~\ref{thm:main}]
 It is proved in~\cite[Proposition~4.4]{Mischler2016} that for any $k>1$ and any $a>(2^{1-k}-1)b$, there exist a nonnegative integer $M$,
 a finite family of distinct complex numbers $\xi_0,\cdots,\xi_M\in\overline{\Delta_a}$, some finite rank projectors $\Pi_0,\cdots,\Pi_M$ and some nonnegative integers $j_0,\cdots,j_M$ such that
 \begin{equation}\label{eq:Weyl}
 \bigg\| S_t - \sum_{m=0}^M \e^{\xi_m t} \sum_{j=0}^{j_m}\frac{t^j}{j!}(\scrL-\xi_m)^j\Pi_m \bigg\|\leq C_a\,\e^{at}
 \end{equation}
 for some $C_a>0$ and all $t\geq0$.
 The pointwise expansion~\eqref{eq:expansion} then imposes $M=m_a$, $\{\xi_0,\cdots,\xi_M\}=\{\lambda_0,\cdots,\lambda_{m_a}\}$,
 $j_m=0$ for all $m$, and the range of $\Pi_m$ is one dimensional, given by $R\Pi_m=\Span(f_m)$.
 This implies that $\Pi_m=\psi_m\otimes f_m$ for some $\psi_m\in(L^1_k)'=L^\infty_k$.
 In other words we have for any $u_0\in L^1_k$ and all $t\geq0$
 \begin{equation}\label{eq:higher-order-bis}
 \bigg\| S_tu_0 - \sum_{m=0}^{m_a} \e^{\lambda_m t} \langle\psi_m,u_0\rangle f_m \bigg\|_{L^1_k}\leq C_a\,\e^{at}{\|u_0\|}_{L^1_k}.
 \end{equation}
Testing the term inside the norm of the left hand side against~$\phi_\ell$ for some $\ell\in\{0,\cdots,m_a\}$, using~\eqref{eq:orthonorm} and dividing by $ \e^{\lambda_\ell t} $, we get that
\[|\langle\phi_\ell,u_0\rangle - \langle\psi_\ell,u_0\rangle| \leq C_a\,\e^{(a-\lambda_\ell) t}{\|\phi_\ell\|}_{L^\infty_k}{\|u_0\|}_{L^1_k}\xrightarrow[t\to+\infty]{}0\]
for any $u_0\in L^1_k$, since $a<\lambda_\ell$ for $\ell\leq m_a$.
This implies that $\psi_\ell=\phi_\ell$ and the proof is complete, \eqref{eq:higher-order} being obtained by applying~\eqref{eq:higher-order-bis} to the initial data
$u_0-\sum_{m=0}^{m_a}\langle \phi_m,u_0\rangle f_m$, using again~\eqref{eq:orthonorm}.
 \end{proof}

 \section{Conclusion and perspectives}
 
 We have proved that for any $a>b$, the part of spectrum of $\scrL$ in $L^1_k$ with real part larger than $a$ reduces to a finite number of explicit algebraically simple eigenvalues with explicit associated eigenfunctions,
 provided that $k$ is large enough (larger than an explicit bound).
 
 \medskip 

 The existence of explicit higher eigenfunctions $f_m$ has also been obtained in the literature for more general versions of Equation~\eqref{eq:mitosis},
 namely with a monomial size-dependent division rate
\begin{equation}\label{eq:monomial}
 \partial_tu(t,x)+g\,\partial_xu(t,x)+bx^nu(t,x)=2^{n+2}bx^nu(t,2x)
 \end{equation}
in~\cite{vBVH1}, or with non-symmetric division
\begin{equation}\label{eq:asymmetric}
 \partial_tu(t,x)+g\,\partial_xu(t,x)+b\,u(t,x)=\alpha b\,u(t,\alpha x) + \beta b\,u(t,\beta x) 
 \end{equation}
where $1/\alpha+1/\beta=1$ in~\cite{Zaidi2015a}.
Extending Theorem~\ref{thm:main} to these more general model is an interesting open problem.
It would require deriving the asymptotic expansion~\eqref{eq:expansion}, which is crucial in the proof, to the solutions of~\eqref{eq:monomial} and/or~\eqref{eq:asymmetric},
or finding another way to identify $\xi_m,j_m$ and $\Pi_m$ in~\eqref{eq:Weyl}.
This is left for future work.

\

\paragraph{\bf Acknowledgments.}
PG has been supported by the ANR project NOLO (ANR-20-CE40-0015), funded by the French Ministry of Research.

This work was partially funded by the European Union (ERC, SINGER, 101054787). Views and opinions expressed are however those of the author(s) only and do not necessarily reflect those of the European Union or the European Research Council. Neither the European Union nor the granting authority can be held responsible for them.

%
%


\end{document}